\documentclass[12pt,reqno,makeidx]{amsart}%

\usepackage[margin=1in]{geometry}
\usepackage[T1]{fontenc}
\usepackage{lmodern}
\usepackage{microtype}
\usepackage{amsmath,amssymb,mathtools}
\usepackage{xcolor}
\usepackage[shortlabels]{enumitem}
\usepackage[colorlinks,linkcolor=magenta,citecolor=blue,urlcolor=blue]{hyperref}

\numberwithin{equation}{section}

\usepackage{amssymb,latexsym}
\usepackage{amsmath}
\usepackage{amsfonts}
\usepackage{amssymb}
\usepackage{graphicx}%
\newtheorem{theorem}{Theorem}[section]
\newtheorem{lemma}[theorem]{Lemma}
\newtheorem{remark}[theorem]{Remark}
\newtheorem{corollary}[theorem]{Corollary}
\newtheorem{problem}[theorem]{Problem}

\newtheorem{definition}[theorem]{Definition}
\newtheorem{example}[theorem]{Example}

\makeindex

\newcommand{\cM}[2]{\mathcal M_{{#1},{#2}}}

\def\cMEE{\cM{E_*}{E}}

\def\CM{\mathcal M}

\def\otw{\otimes_{{\mbox{\tiny{wot}}}}}

\begin{document}
\title[The Douglas Property]{The Douglas Property for Multiplier Algebras of Operators}
\author{S. McCullough and T. T. Trent}
\address{Department of Mathematics\\
University of Florida\\
Box 118105\\
Gainesville, FL 32611-8105 }
\email{sam@ufl.edu}
\address{Department of Mathematics\\
University of Alabama\\
Box 870350\\
Tuscaloosa, AL 35487-0350}
\email{ttrent@as.ua.edu}

\baselineskip=15pt

\begin{abstract}
 For a collection of  reproducing kernels $k$ 
 that includes those for the Hardy space
  of the polydisk and the ball and for the
  Bergman space, 
 $k$ is a \emph{ complete Pick} kernel if and only 
 if the multiplier algebra of $H^2(k)$ has the 
 Douglas property. Consequences for solving the 
 operator equation $AX=Y$ are examined.
\end{abstract}
\maketitle

\section{Introduction}
 \label{sec:blah-blah}
  Let $H$ denote a (complex) Hilbert space and let $B(H)$ denote the
  algebra of  bounded
  operators on $H$. Given $A,B\in B(H),$ when does there
  exist $X\in B(H)$ such that $AX=B$ and, if such an $X$ exists,
  what is the smallest possible norm? 
  The solution to both questions is given by the well-known 
  Douglas Lemma \cite{D}, which says that there is an $X$ of norm at most one
  such that $AX=B$ if and only if $AA^* \succeq  BB^*$. 

  Let $E$ denote a Hilbert space. A theorem of Leech \cite{L}
  says that the Douglas Lemma remains true if $B(H)$ is replaced
  by the algebra $\mathcal T_E$ of bounded analytic Toeplitz
  operators on the Hardy space of $E$-valued functions,
  denoted by $H_E^2(\mathbb D)$. Thus, if $T_A$ and $T_B$
  have $B(E)$-valued analytic symbols $A$ and $B$, respectively,
  then there is a bounded analytic Toeplitz operator $T_C$
  with $\|T_C\|\leq 1$ such that $T_AT_C=T_B$ if and only if
  $T_AT_A^*\succeq T_BT_B^*$.

  If $\mathcal A$ is an algebra of operators on a Hilbert space
  $H$, then $M_n(\mathcal A)$, the algebra of $n\times n$ matrices
  with entries in $\mathcal A$, is, in the natural way,
  an algebra of operators on $\oplus_1^n H$, the Hilbert
  space direct sum of $H$ with itself $n$ times. 
  The \emph{Douglas property} for a wot-closed unital algebra
  $\mathcal A$ is defined in Definition~\ref{def:douglas}
  using its countable amplification. In particular, if
  $\mathcal A$ has the Douglas property, then for every $n$
  and every $A,B\in M_n(\mathcal A)$ there exists
  $C\in M_n(\mathcal A)$ with $\|C\|\leq 1$ and $AC=B$
  if and only if $AA^*\succeq BB^*$.
  This finite-matrix consequence is weaker than the definition
  for general wot-closed unital algebras, as shown in
  Section~\ref{app:finite-matrices}.

  \index{Douglas property}  The Douglas
  Lemma and Leech's Theorem say that $B(H)$ and the algebra
  of analytic Toeplitz operators respectively have the Douglas property.
  Fialkow and Salas studied analogous majorization--factorization
  properties within $C^*$-algebras \cite{FS}.
  This article considers the question of
  which multiplier algebras on reproducing kernel
  Hilbert spaces, like $\mathcal T_E$, have
  the Douglas property.  

  A main result of this article, 
  Theorem \ref{thm:main}, says, for a 
  natural collection of reproducing kernels
  $k$,  if the algebra of multipliers 
  on the corresponding reproducing kernel Hilbert space has
  the Douglas property,  then $k$ is a complete
   Pick kernel.  As a consequence it follows that
   the multiplier algebras of 
  the Hardy spaces 
  on the unit 
  ball and the unit polydisk in dimension $n \ge 2$ and the Bergman 
  spaces on the unit ball and the unit polydisk in all dimensions,
  do not have the Douglas property,  since 
  it is well known that the reproducing kernels of these spaces 
  are not complete Pick kernels \cite{M} \cite{Q}. 
  For each multiplier algebra $\mathcal M$ listed in
  Theorem~\ref{thm:main-aux}, there exist
  $A,B\in\mathcal M\otw B(\ell^2)$ such that
  $AA^*\succeq BB^*$ but $AX=B$ has no solution in
  $\mathcal M\otw B(\ell^2)$.
  The tensor product is defined in Subsection~\ref{sec:Douglas}.
  This nonfactorization statement is the other main result of
  the paper. Examples and questions appear in
  Section~\ref{sec:egandq}.

 In the remainder of this introduction we state 
  precisely the main results, first
 introducing the needed definitions and background. 
 Subsection \ref{subsec:mult-algs} discusses reproducing kernel
 Hilbert spaces and their multiplier algebras.
  The Douglas property is discussed in
  further detail in Subsection \ref{sec:Douglas}.  The main results
 are stated in Subsection \ref{sec:main-things}.


\subsection{Reproducing kernels and multiplier algebras}
 \label{subsec:mult-algs}
 Let $\Omega$ denote a set, which in applications is 
 generally a bounded domain in $\mathbb C^d$. 
 A positive semi-definite function,
  or kernel, $k:\Omega\times \Omega\to\mathbb C,$
 \index{kernel} \index{positive semi-definite function}
 determines, by standard constructions, a Hilbert space $H^2(k)$
  of functions $f:\Omega\to \mathbb C$.  In particular,
   for each $w\in\Omega$ the function $k(\cdot,w)\in H^2(k)$
   reproduces the value of  an $f\in H^2(k)$ at $w$; i.e.,
\[
  f(w)=\langle f,k(\cdot,w) \rangle.
\]    
  Thus, $\langle k(\cdot,w),k(\cdot,z)\rangle = k(z,w)$ and 
  the span of $\{k(\cdot,w):w\in\Omega\}$ is dense in
  $H^2(k)$. \index{reproducing kernel Hilbert space}
  There is little lost by assuming, as we generally will,
  that $k(z,z)>0$ for all $z\in\Omega$. 


  The multipliers \index{multiplier}
   of $H^2(k)$ are those functions $\phi:\Omega\to\mathbb C$
  such that $\phi h\in H^2(k)$ for every $h\in H^2(k)$.
   By the closed graph theorem $\phi$ then determines a bounded operator
   $M_\phi$ on $H^2(k)$ defined by $M_\phi h=\phi h$. 
   Let $\mathcal M(k)$ denote the set of multipliers of $H^2(k)$ 
   identified as the unital subalgebra $\{M_\phi:\phi \in \mathcal M(k)\}$
  of $B(H^2(k))$. \index{Multiplier algebra}
 For example, the Hardy space $H^2(\mathbb{D})$ is a reproducing
  kernel Hilbert space whose kernel $s$ is the Szeg\"o kernel
 \index{Szeg\"o kernel}
\[
  s(z,w) = \frac 1{1- z\overline w}.
\]
 In this case 
  the multiplier algebra $\mathcal{M}(s)$
  is $H^{\infty}(\mathbb{D})$,  
  the algebra  of bounded analytic functions on the unit disk.

\begin{definition}\rm
 \label{def:multalgEEstar}\index{$\cMEE$}\index{multiplier algebra}
  More generally, given Hilbert spaces $E$ and $E_*$,
  let $\cMEE$ denote the corresponding multipliers; i.e.,  those
  functions $\Phi:\Omega \to B(E_*,E)$ such that $\Phi H\in H^2(k)\otimes E$
   for every $H\in H^2(k)\otimes E_*$. 
\end{definition}

    Observe,
    if $e,e_*$ are in $E$ and $E_*$ respectively, then
    $\phi(w)=\langle \Phi(w)e_*,e\rangle$ is in $\mathcal M(k)$.
    Further, if $\Phi \in \cM{E}{F}$ and $\Psi \in \cM{F}{G}$,
    then $\Psi\Phi \in \cM{E}{G}$.

\begin{definition}\rm
 \label{def:nice} \index{nice} \index{kernel, nice}
  We say that a reproducing kernel $k$ is {\it nice} if
  the Hilbert space $H^2(k)$ 
  is separable and 
  there exist $p,q\in \cM{{\ell^2}}{{\mathbb C}}$ such that
\begin{equation}
 \label{eq:oneoverk}
    1= k(z,w)[p(z)p(w)^*-q(z)q(w)^*]
\end{equation}
   for all $z,w\in\Omega.$ 
\end{definition}

 Of course, if $k$ is nice, then $k(z,w)$ is never zero.

 We close this subsection by recalling 
 the notion of a complete Pick  kernel \cite{AM1}.

\begin{definition}\rm
 \label{def:NPk} \index{complete Pick kernel} \index{NP kernel} 
  Suppose $k$ is a positive semi-definite function on $\Omega.$ 
   The kernel $k$ is
  a {\it complete Pick kernel,} an {\it NP kernel}
    for short, if for each $\omega\in\Omega$ 
  there exists a 
  positive semi-definite function $L_\omega:\Omega\times \Omega\to \mathbb C$
   so that
\begin{equation}
 \label{eq:np}
  k(y,x)k(\omega,\omega)-k(y,\omega)k(\omega,x)=L_\omega(y,x)k(y,x).
\end{equation}
\end{definition}

\begin{remark}\rm
 \label{rem:whyNP}
   The reason for the names Pick and NP kernel can be found in \cite{AM2}.
  See also \cite{M} and \cite{Q}.
\end{remark}

\begin{remark}\rm
 \label{rem:base-no-matter}
   If $k(z,w)$ never vanishes and  if equation \eqref{eq:np}
   holds for one $\omega$, then  it holds for all
   $\omega$  and thus $k$
   is an NP kernel.   See \cite{MT} for details.  
\end{remark}

\begin{remark}\rm
 \label{rem:nozeros}
  By standard reproducing kernel arguments, the positive semi-definite 
  function $L_\omega$ can be factored as $B(w)^*B(z)$, where
   $B:\Omega\to \mathcal E$,  for some auxiliary Hilbert space
  $\mathcal E$.  When, as in all the examples in this article, $H^2(k)$ is
  separable, $\mathcal E$ can be chosen separable.  In that case, choosing
  a basis $\{e_j\}$ for $\mathcal E$ and letting 
   $b_j(z)= \langle B(z),e_j\rangle$ it follows that 
 \[
   L_\omega(y,x) = \sum b_j(y) b_j(x)^*.
 \]
   
\end{remark}

\subsection{The Douglas Property}
 \label{sec:Douglas}
   Given Hilbert spaces $H$ and $K$ and 
   operators $A\in B(H)$ and $B\in B(K)$, 
   the tensor product $A\otimes B$
   is the operator on the Hilbert space
   $H\otimes K$ determined
   by its action on elementary tensors,
 \[
    A\otimes B (h\otimes f) = Ah\otimes Bf.
 \] 
  It can be verified that $A\otimes B$ is bounded. In 
  fact $\|A\otimes B\|=\|A\|\, \|B\|$. 

  As an example, if $k$ is a kernel, $\varphi\in \CM(k),$
  and $B\in B(K)$, then $\Phi(z)=\varphi(z)B$ is 
  in $\cM{K}{K}$ and corresponds to the operator
  $M_\Phi = M_\varphi\otimes B$. 

\begin{definition}\rm
 \label{def:tensor-product} \index{$\mathcal A\otimes B(\ell^2)$}
  \index{tensor product}
  Given a unital subalgebra $\mathcal A$ of $B(H)$, 
  let $\mathcal A\otimes B(K)$ denote the 
  algebraic tensor product; i.e., finite sums
  $\sum_1^n A_j\otimes B_j$.  Let $\mathcal A\otw B(K)$
  denote the closure of the algebraic tensor product in
  the weak operator topology (wot) of $B(H\otimes K)$.

\end{definition}


\begin{definition}
 \label{def:douglas} \index{Douglas property}
  A wot-closed  
   unital subalgebra $\mathcal{A}$ of $B(H)$ 
  has the Douglas property if, whenever 
  $A, B \in \mathcal{A} \otw B(\ell^2)$ and
\[
  A\, A^* \succeq B\, B^*,
\]
then there exists
\[
  C \in \mathcal{A}\otw B(\ell^2)
\]
  such that $AC=B$ and $\|C\|\le 1$. 
\end{definition}

  The Douglas property implies the corresponding factorization
  property for finite matrices: if
  $A\in M_{r,s}(\mathcal A)$ and
  $B\in M_{r,t}(\mathcal A)$ satisfy $AA^*\succeq BB^*$,
  then there exists $C\in M_{s,t}(\mathcal A)$ such that
  $AC=B$ and $\|C\|\leq 1$.
  This follows by embedding the finite matrices as corners of
  $\mathcal A\otw B(\ell^2)$ and compressing a contractive solution.
  The converse fails for general wot-closed unital operator
  algebras; see Section~\ref{app:finite-matrices}.  It does,
  however, hold for a number of cases of interest; see
  Section~\ref{sec:finite-countable}.
  Throughout this paper, the term Douglas property refers to
  Definition~\ref{def:douglas}.

  The following standard lemma says that it makes sense to ask if
  the multiplier algebra $\CM(k)$ corresponding to a kernel $k$
  has the Douglas Property.
  Note that 
\[
   \CM(k)\otimes B(\ell^2)\subset \cM{\ell^2}{\ell^2} 
     \subset B(H^2(k) \otimes \ell^2).
\]

\begin{lemma}
 \label{lem:op-mults}
   If $k$ is a reproducing kernel and $k(z,z)>0$ for all
   $z\in\Omega$, then 
   the algebra $\CM(k)\subset B(H^2(k))$ is wot-closed and moreover
   $\cM{\ell^2}{\ell^2} = \CM(k) \otw B(\ell^2)$.
\end{lemma}

  The proof appears in Section \ref{sec:preliminaries}. 

\subsection{Main Results}
 \label{sec:main-things}
  The following is our  main result on multiplier algebras with the 
  Douglas property.

\begin{theorem}
  \label{thm:main} 
    Suppose $k$ is a nice reproducing kernel over the set $\Omega$.
   If $\CM(k)$ has the Douglas property, then $k$ is a
   complete Pick kernel.

   Conversely, if $k$ is a non-vanishing complete
  Pick kernel, then $\CM(k)$ has the
  Douglas property. 
\end{theorem}


  The converse direction in Theorem \ref{thm:main} is a result
  from \cite{BTV}. 

   Theorem \ref{thm:main} applies to some favorite examples.

\begin{corollary}
 \label{cor:Douglas-not}
  The multiplier algebras 
  for each of the spaces $A^2(\mathbb B^m),$  $H^2(\mathbb{D}^n)$, 
  and $H^2(\mathbb B^n)$, 
  for $m \ge 1$ and $n \ge 2$ do not have the Douglas property.

  Here $A^2(\mathbb B^m)$ is the Bergman space of the
  unit ball $\mathbb B^m$ in $\mathbb C^m$; 
   $H^2(\mathbb D^n)$ is the Hardy space of the polydisk
 $\mathbb D^n$ in $\mathbb C^n$; and $H^2(\mathbb B^n)$
  is the Hardy space of the ball. 
\end{corollary}

\begin{proof}
   It is clear that these are nice
   reproducing kernel Hilbert spaces. 
   Further, it is well known, and easy to verify,
   that their respective  kernels are not complete Pick kernels. 
\end{proof}

  It turns out that without the Douglas property it is not always
  possible to factor, even dropping the norm constraint.

\begin{theorem}
 \label{thm:main-aux}
  Let $ \mathcal{A}$ denote  the multiplier algebra on any of 
  the Hilbert spaces
  $A^2(\mathbb B^m)$, $H^2(\mathbb D^n)$, and
  $H^2(\mathbb B^n)$ for $m\geq1$ and $n\geq2$.
  There exist $A,B\in\mathcal A\otw B(\ell^2)$ such that
  $AA^*\succeq BB^*$ but the equation $AX=B$ has no solution
  $X\in\mathcal A\otw B(\ell^2)$.

\end{theorem}


   
  The next section contains routine, but necessary, preliminary results.
  The proofs of Theorems \ref{thm:main} and \ref{thm:main-aux}
  occupy Sections \ref{sec:main} and \ref{sec:main-aux} 
  respectively. 
  Examples and questions appear in Section~\ref{sec:egandq}.
  
  This revision was undertaken to remove the false claim
  appearing between Definition~\ref{def:douglas} and Lemma~\ref{lem:op-mults}
  in the version posted to the arXiv in November 2011 that it suffices to check
  the Douglas property for  $\mathcal{A}\otimes M_n(\mathbb{C})$ for each $n.$ 
  The authors thank Daniel Alpay for questioning this claim.
  A counterexample distinguishing finite-matrix factorization from the Douglas property
  appears in Section~\ref{app:finite-matrices}.  However, in many cases
  of interest here, finite matrix factorizations suffice. Section~\ref{sec:finite-countable} gives the full
  Douglas property for complete Pick kernels and shows that, for
  nice kernels whose row $p$ can be chosen with finitely many
  components, finite-matrix factorization, the complete Pick
  property, and the Douglas property are equivalent.
  These arguments do not settle the corresponding implication
  for arbitrary commutative wot-closed unital algebras;
  see Remark~\ref{rem:finite-commutative}.
  Appendix~\ref{app:changes} records the substantive changes
  from the version posted to the arXiv in November of 2011.  That version ended with Section~\ref{sec:egandq} and,
  to that point, this version retains the original numbering.
  
   The authors acknowledge use of LLMs in the preparation
  of this revision. The LLMs  provided the counterexample in Section~\ref{app:finite-matrices}
  and provided major assistance with the results in Section~\ref{sec:finite-countable}. They
  also corrected numerous typos and small inconsistencies.

\section{Preliminary Results}
 \label{sec:preliminaries}
  This section collects a few preliminary observations used in the
  proofs of Theorems~\ref{thm:main} and~\ref{thm:main-aux}.

\begin{lemma}
 \label{lem:eigs-adjoints}
  If $\Phi\in\cM{E_*}{E}$, $e\in E$, and $w\in\Omega$, then
 \[
  M_\Phi^*(k(\cdot,w)\otimes e)
    =k(\cdot,w)\otimes\Phi(w)^*e.
 \]
\end{lemma}

\begin{proof}
  Given $F\in H^2(k)\otimes E_*$, we have
 \[
 \begin{split}
  \langle F,M_\Phi^*(k(\cdot,w)\otimes e)\rangle
   &=\langle \Phi F,k(\cdot,w)\otimes e\rangle\\
   &=\langle \Phi(w)F(w),e\rangle\\
   &=\langle F(w),\Phi(w)^*e\rangle\\
   &=\langle F,k(\cdot,w)\otimes\Phi(w)^*e\rangle.
 \end{split}
 \]
\end{proof}

  The following is a slight generalization of Lemma~\ref{lem:op-mults}.

\begin{lemma}
 \label{lem:op-mults-plus}
  Suppose $k$ is a reproducing kernel with $k(z,z)>0$ for every
  $z\in\Omega$, and let $E$ and $E_*$ be separable Hilbert spaces.
  Then
 \[
  \cM{E_*}{E}=\CM(k)\otw B(E_*,E),
 \]
  where the right side is the wot closure of the algebraic tensor
  product in $B(H^2(k)\otimes E_*,H^2(k)\otimes E)$.
\end{lemma}

\begin{proof}
  We first show that $\cM{E_*}{E}$ is wot-closed.
  Suppose that a net of multiplication operators $M_{\Phi_\alpha}$
  converges in the weak operator topology to an operator
  $T:H^2(k)\otimes E_*\to H^2(k)\otimes E$.
  Fix $z\in\Omega$, and put $k_z=k(\cdot,z)$. Define
  $\Phi(z)\in B(E_*,E)$ by
 \[
  \langle\Phi(z)e_*,e\rangle
   =\frac{1}{k(z,z)}
      \langle T(k_z\otimes e_*),k_z\otimes e\rangle.
 \]
  This defines a bounded operator with $\|\Phi(z)\|\leq\|T\|$,
  and Lemma~\ref{lem:eigs-adjoints} shows that
  $\Phi_\alpha(z)\to\Phi(z)$ in the weak operator topology.
  For $g\in H^2(k)$, $e_*\in E_*$, and $e\in E$, it follows that
 \[
 \begin{split}
  \langle T(g\otimes e_*),k_z\otimes e\rangle
   &=\lim_\alpha
       \langle M_{\Phi_\alpha}(g\otimes e_*),k_z\otimes e\rangle\\
   &=g(z)\langle\Phi(z)e_*,e\rangle.
 \end{split}
 \]
  Thus $T=M_\Phi$, first on elementary tensors and then, by
  continuity of point evaluations, on all of $H^2(k)\otimes E_*$.
  This proves wot-closedness.

  The space $\cM{E_*}{E}$ contains the algebraic tensor product
  $\CM(k)\otimes B(E_*,E)$, and therefore contains its wot closure.
  Conversely, there exist sequences of finite-rank projections
  $P_n$ on $E$ and $Q_n$ on $E_*$ converging strongly to the
  respective identity operators. For $\Phi\in\cM{E_*}{E}$,
 \[
  (I\otimes P_n)M_\Phi(I\otimes Q_n)
 \]
  belongs to $\CM(k)\otimes B(E_*,E)$ for every $n$, because
  every scalar matrix coefficient of $\Phi$ is in $\CM(k)$.
  These operators converge strongly, and hence in wot, to $M_\Phi$.
  This gives the reverse inclusion.
\end{proof}

\section{The proof of Theorem \ref{thm:main}}
 \label{sec:main}
  This section proves the forward implication of
  Theorem~\ref{thm:main}; the converse is the result from
  \cite{BTV} cited above.

  The assumption that $k$ is nice gives
  $p,q\in\cM{\ell^2}{\mathbb C}$ satisfying
  equation~\eqref{eq:oneoverk}. By Lemma~\ref{lem:eigs-adjoints},
 \[
  \langle M_q^*k(\cdot,y),M_q^*k(\cdot,x)\rangle
    =q(x)q(y)^*k(x,y).
 \]
  Consequently,
 \begin{equation}
  \label{eq:nice}
  \langle (M_pM_p^*-M_qM_q^*)k(\cdot,y),k(\cdot,x)\rangle=1,
 \end{equation}
  and hence $M_pM_p^*-M_qM_q^*\succeq0$.
  Embed $M_p$ and $M_q$ as the first rows of operators on
  $H^2(k)\otimes\ell^2$, with all other rows zero.
  Lemma~\ref{lem:op-mults-plus} and
  Definition~\ref{def:douglas} then give
  $C\in\cM{\ell^2}{\ell^2}$ with
 \[
  q=pC,\qquad M_q=M_pM_C,\qquad \|M_C\|\leq1.
 \]
  In particular, this application uses Definition~\ref{def:douglas}
  directly and does not require a converse to its finite-matrix
  consequence. The remainder of the proof exploits the identity
 \[
  M_pM_p^*-M_qM_q^*=M_p(I-M_CM_C^*)M_p^*.
 \]

  Fix $\omega\in\Omega$. Let $H^2_\omega(k)$ be the subspace of
  functions in $H^2(k)$ that vanish at $\omega$, let $P_\omega$
  be the orthogonal projection onto this subspace, and set
 \[
  Q_\omega=P_\omega\otimes I_{\ell^2}.
 \]
  Since $\|M_C\|\leq1$, the operator
 \[
  D=P_\omega M_pQ_\omega
       (I-M_CQ_\omega M_C^*)
       Q_\omega M_p^*P_\omega
 \]
  is positive semi-definite. Thus
 \[
  L_\omega(x,y)
    =\langle Dk(\cdot,y),k(\cdot,x)\rangle
 \]
  is a positive semi-definite kernel.

  Observe that
 \begin{equation}
  \label{eq:1}
  P_\omega k(\cdot,y)
    =k(\cdot,y)
      -\frac{k(\omega,y)}{k(\omega,\omega)}k(\cdot,\omega).
 \end{equation}
  Lemma~\ref{lem:eigs-adjoints} and~\eqref{eq:1} give
 \begin{equation}
  \label{eq:L1}
  Q_\omega M_p^*P_\omega k(\cdot,y)
    =P_\omega k(\cdot,y)\otimes p(y)^*
 \end{equation}
  and
 \begin{equation}
  \label{eq:L2}
 \begin{split}
  Q_\omega M_C^*Q_\omega(k(\cdot,y)\otimes p(y)^*)
    &=P_\omega k(\cdot,y)\otimes C(y)^*p(y)^*\\
    &=P_\omega k(\cdot,y)\otimes q(y)^*.
 \end{split}
 \end{equation}
  Here $p(y)^*$ and $q(y)^*$ are identified with vectors in
  $\ell^2$. Combining these identities yields
 \[
  L_\omega(x,y)
    =[p(x)p(y)^*-q(x)q(y)^*]\,
      \langle P_\omega k(\cdot,y),P_\omega k(\cdot,x)\rangle.
 \]
  By~\eqref{eq:oneoverk},
 \begin{equation}
  \label{eq:2}
  k(x,y)L_\omega(x,y)
    =\langle P_\omega k(\cdot,y),P_\omega k(\cdot,x)\rangle.
 \end{equation}
  Equations~\eqref{eq:1} and~\eqref{eq:2} now give
 \[
  k(x,y)L_\omega(x,y)
    =k(x,y)-\frac{k(x,\omega)k(\omega,y)}{k(\omega,\omega)}.
 \]
  Define the positive semi-definite kernel
 \[
  \widetilde L_\omega=k(\omega,\omega)L_\omega.
 \]
  It satisfies~\eqref{eq:np}. Since $\omega$ was arbitrary,
  $k$ is a complete Pick kernel.

\section{The Proof of Theorem \ref{thm:main-aux}}
 \label{sec:main-aux}
 Theorem \ref{thm:main-aux} concerns the Bergman spaces of the
 balls $\mathbb B^m$, the Hardy spaces of the balls $\mathbb B^m$
 for $m\ge 2$, and the Hardy spaces of the polydisks $\mathbb D^m$
 for $m\ge 2$. We begin with three elementary identities for the
 Bergman disk, the Hardy bidisk, and the Hardy two-dimensional ball.
 The proof then treats all the stated dimensions.

 Given an indexed family $\{h_j:j\in J\}$ in a Hilbert space,
 let $[h_j:j\in J]$ denote its closed linear span and let
 $\mathrm{Proj}[h_j:j\in J]$ denote the orthogonal projection onto
 that span. To distinguish rank-one operators from tensor products,
 write
 \[
   \theta_{f,g}h=\langle h,g\rangle f.
 \]
 We normalize the measures defining the Hardy and Bergman norms
 so that the constant function $1$ has norm one.

\begin{lemma}
 \label{lem:bergman}
 Let $B=M_z$ on $A^2(\mathbb D)$, the Bergman space of the unit disk.
 For $N=1,2,\ldots$,
 \[
 \begin{split}
 I+NB^{N+1}B^{*(N+1)}-(N+1)B^NB^{*N}
   &=\mathrm{Proj}[1,z,\ldots,z^{N-1}]\\
   &=\sum_{j=0}^{N-1}(j+1)\theta_{z^j,z^j}.
 \end{split}
 \]
\end{lemma}

\begin{proof}
 The reproducing kernel is $k(z,w)=(1-z\overline w)^{-2}$.
 Evaluating the matrix coefficients of the operators on kernel
 functions reduces the assertion to the identity
 \begin{equation}
  \label{eq:bergman-polynomial}
  \frac{1+Nx^{N+1}-(N+1)x^N}{(1-x)^2}
     =\sum_{j=0}^{N-1}(j+1)x^j,
     \qquad |x|<1.
 \end{equation}
 Indeed,
 \[
 \begin{split}
 \frac{1+Nx^{N+1}-(N+1)x^N}{(1-x)^2}
  &=\frac{\sum_{j=0}^{N-1}x^j-Nx^N}{1-x}\\
  &=\sum_{j=0}^{N-1}x^j\frac{1-x^{N-j}}{1-x}\\
  &=\sum_{j=0}^{N-1}x^j\sum_{r=0}^{N-1-j}x^r
   =\sum_{j=0}^{N-1}(j+1)x^j.
 \end{split}
 \]
 Finally, the monomials are orthogonal in $A^2(\mathbb D)$ and
 $\|z^j\|^2=1/(j+1)$, which identifies the indicated sum of
 rank-one operators as the orthogonal projection.
\end{proof}

\begin{lemma}
 \label{lem:bidisk}
 Let $S$ and $W$ denote multiplication by $z$ and $w$, respectively,
 on $H^2(\mathbb D^2)$. For $N=1,2,\ldots$,
 \[
 \begin{split}
 I+\sum_{j=1}^N S^{N-j+1}W^jW^{*j}S^{*(N-j+1)}
 &-\sum_{j=0}^N S^{N-j}W^jW^{*j}S^{*(N-j)}\\
 &=\mathrm{Proj}[z^pw^q:p,q\ge0,\ p+q<N]\\
 &=\sum_{q=0}^{N-1}\sum_{p=0}^{N-1-q}
       \theta_{z^pw^q,z^pw^q}.
 \end{split}
 \]
\end{lemma}

\begin{proof}
 The monomials form an orthonormal basis of $H^2(\mathbb D^2)$.
 Using its reproducing kernel, the required operator identity
 reduces to
 \[
 \frac{1+\sum_{j=1}^N x^{N-j+1}y^j
           -\sum_{j=0}^N x^{N-j}y^j}{(1-x)(1-y)}
   =\sum_{p=0}^{N-1}\sum_{q=0}^{N-1-p}x^py^q,
 \qquad |x|,|y|<1.
 \]
 This follows from
 \[
 \begin{split}
 &\frac{1+\sum_{j=1}^N x^{N-j+1}y^j
           -\sum_{j=0}^N x^{N-j}y^j}{(1-x)(1-y)}\\
 &\quad=\frac{(1-x^N)-(1-x)\sum_{j=1}^N x^{N-j}y^j}
                  {(1-x)(1-y)}\\
 &\quad=\frac{\sum_{p=0}^{N-1}x^p
                  -\sum_{p=0}^{N-1}x^py^{N-p}}{1-y}\\
 &\quad=\sum_{p=0}^{N-1}x^p\sum_{q=0}^{N-1-p}y^q.
 \end{split}
 \]
\end{proof}

\begin{lemma}
 \label{lem:ball}
 Let $S$ and $W$ denote multiplication by $z$ and $w$ on
 $H^2(\mathbb B^2)$. For $N=1,2,\ldots$,
 \[
 \begin{split}
 I&+N\sum_{j=0}^{N+1}\binom{N+1}{j}
        S^{N+1-j}W^jW^{*j}S^{*(N+1-j)}\\
 &\succeq(N+1)\sum_{j=0}^N\binom Nj
        S^{N-j}W^jW^{*j}S^{*(N-j)}.
 \end{split}
 \]
\end{lemma}

\begin{proof}
 Let $P_N$ be the orthogonal projection of $H^2(\mathbb B^2)$
 onto the span of the monomials of total degree less than $N$.
 In fact,
 \[
 \begin{split}
 I&+N\sum_{j=0}^{N+1}\binom{N+1}{j}
        S^{N+1-j}W^jW^{*j}S^{*(N+1-j)}\\
 &\quad-(N+1)\sum_{j=0}^N\binom Nj
        S^{N-j}W^jW^{*j}S^{*(N-j)}=P_N.
 \end{split}
 \]
 To see this, set $t=z\overline u+w\overline v$ for points
 $(z,w),(u,v)\in\mathbb B^2$. The kernel of $H^2(\mathbb B^2)$
 is $(1-t)^{-2}$. By the binomial theorem and
 \eqref{eq:bergman-polynomial}, the matrix coefficient of the
 operator on the left, evaluated on the two kernel functions, is
 \[
   \frac{1+Nt^{N+1}-(N+1)t^N}{(1-t)^2}
      =\sum_{j=0}^{N-1}(j+1)t^j.
 \]
 Since $(j+1)t^j$ is the reproducing kernel of the homogeneous
 polynomials of degree $j$ in $H^2(\mathbb B^2)$, the last sum
 is the kernel of $P_N$. This proves the identity and hence
 the inequality. For $N=1$ the identity reads
 \[
 I+S^2S^{*2}+2SWW^*S^*+W^2W^{*2}
     =2SS^*+2WW^*+\theta_{1,1}.
 \]
\end{proof}

\begin{proof}[Proof of Theorem \ref{thm:main-aux}]
 We first consider $\mathcal H=H^2(\mathbb D^2)$, whose
 multiplier algebra is $H^\infty(\mathbb D^2)$. For each $N\ge1$,
 let $A_N,B_N\in M_{N+1}(H^\infty(\mathbb D^2))$ have first rows
 \[
   [I,S^NW,S^{N-1}W^2,\ldots,SW^N]
   \quad\hbox{and}\quad
   [S^N,S^{N-1}W,\ldots,W^N],
 \]
 respectively, with every other row zero. Lemma \ref{lem:bidisk}
 gives $A_NA_N^*\succeq B_NB_N^*$.

 Suppose a matrix multiplier $C=[C_{jk}]_{j,k=1}^{N+1}$ satisfies
 $A_NM_C=B_N$. For $1\le k\le N+1$, the corresponding identity
 of functions is
 \[
 C_{1k}(z,w)+\sum_{j=1}^N z^{N-j+1}w^j C_{j+1,k}(z,w)
     =z^{N-k+1}w^{k-1}.
 \]
 Every term in the sum on the left has total degree at least
 $N+1$. Thus the coefficient of $z^{N-k+1}w^{k-1}$ in
 $C_{1k}$ is $1$. Consequently,
 \[
 \begin{split}
 N+1
 &\le\sum_{k=1}^{N+1}\|C_{1k}\|_{H^2(\mathbb D^2)}^2\\
 &\le\sup_{(z,w)\in\mathbb D^2}
             \sum_{k=1}^{N+1}|C_{1k}(z,w)|^2\\
 &\le\sup_{(z,w)\in\mathbb D^2}
             \|C(z,w)\|_{B(\mathbb C^{N+1})}^2
 \le\|M_C\|^2.
 \end{split}
 \]
 Here the second inequality follows by integration over concentric
 tori and passage to the Hardy norm, and the last follows from
 Lemma \ref{lem:eigs-adjoints}. In particular, every such solution
 satisfies $\|M_C\|\ge\sqrt{N+1}$.

 Use the same positive scalar $\|A_N\|$ to normalize both matrices,
 and define
 \[
   A=\bigoplus_{N=1}^{\infty}\frac{A_N}{\|A_N\|},
   \qquad
   B=\bigoplus_{N=1}^{\infty}\frac{B_N}{\|A_N\|}
 \]
 on
 \[
   \bigoplus_{N=1}^{\infty}
      \bigl(\mathcal H\otimes\mathbb C^{N+1}\bigr)
      \cong\mathcal H\otimes\ell^2.
 \]
 Since $\|B_N\|\le\|A_N\|$, both direct sums are contractions.
 Their finite partial sums belong to
 $H^\infty(\mathbb D^2)\otimes B(\ell^2)$ and converge strongly
 to $A$ and $B$. Hence
 \[
   A,B\in H^\infty(\mathbb D^2)\otw B(\ell^2),
   \qquad AA^*\succeq BB^*.
 \]
 If $X\in H^\infty(\mathbb D^2)\otw B(\ell^2)$ satisfied
 $AX=B$, its diagonal block $X_{NN}$ relative to this direct sum
 would be a matrix multiplier satisfying $A_NX_{NN}=B_N$.
 It would follow that
 \[
   \|X\|\ge\|X_{NN}\|\ge\sqrt{N+1}
   \qquad\hbox{for every }N,
 \]
 a contradiction.

 For $H^2(\mathbb D^d)$ with $d>2$, use the same matrices in the
 first two coordinate multipliers. The identification
 \[
   H^2(\mathbb D^d)
      =H^2(\mathbb D^2)\otimes H^2(\mathbb D^{d-2})
 \]
 shows that the operator inequality persists. For any matrix
 multiplier solution, restrict its entries to
 $(z,w,0,\ldots,0)$. The preceding coefficient and Hardy-norm
 argument on the bidisk, together with the pointwise bound
 $\|C(z,w,0,\ldots,0)\|\le\|M_C\|$, again gives
 $\|M_C\|\ge\sqrt{N+1}$. The same direct-sum construction
 proves the assertion for every such polydisk.

 We now treat the ball spaces. It is convenient to consider the
 kernels
 \[
   k_\alpha(z,w)=(1-\langle z,w\rangle)^{-\alpha},
   \qquad z,w\in\mathbb B^d,
 \]
 for integers $\alpha\ge2$. The Bergman space of $\mathbb B^d$
 corresponds to $\alpha=d+1$, and the Hardy space of
 $\mathbb B^d$, for $d\ge2$, corresponds to $\alpha=d$.
 Let $S=M_{z_1}$ on $H^2(k_\alpha)$. We claim that
 \begin{equation}
  \label{eq:ball-coordinate-inequality}
  I+NS^{N+1}S^{*(N+1)}\succeq(N+1)S^NS^{*N}
  \qquad(N\ge1).
 \end{equation}
 For this purpose, put
 \[
   u=z_1\overline{w_1},\qquad
   v=\sum_{j=2}^d z_j\overline{w_j},\qquad
   F_N(u)=1+Nu^{N+1}-(N+1)u^N,
 \]
 with $v=0$ when $d=1$. Cauchy--Schwarz gives
 \[
   |u|+|v|\le\sum_{j=1}^d|z_j|\,|w_j|
      \le\|z\|\,\|w\|<1.
 \]
 The binomial series therefore converges absolutely, and
 \eqref{eq:bergman-polynomial} yields
 \[
 \begin{split}
 F_N(u)k_\alpha(z,w)
 &=\sum_{r=0}^{\infty}\binom{\alpha+r-1}{r}
        v^r\frac{F_N(u)}{(1-u)^{\alpha+r}}\\
 &=\sum_{r=0}^{\infty}\binom{\alpha+r-1}{r}
        v^r\frac{\displaystyle\sum_{j=0}^{N-1}(j+1)u^j}
                       {(1-u)^{\alpha+r-2}}.
 \end{split}
 \]
 Each summand is a positive semi-definite kernel: $u$ and $v$
 are positive semi-definite kernels, their products and
 nonnegative integer powers preserve positivity, and the power
 series of $(1-u)^{-(\alpha+r-2)}$ has nonnegative coefficients.
 For exponent zero this last factor is $1$. The convergent sum
 is therefore positive semi-definite. By
 Lemma \ref{lem:eigs-adjoints}, this proves
 \eqref{eq:ball-coordinate-inequality}.

 On $H^2(k_\alpha)\otimes\mathbb C^2$, define
 \[
   A_N=\begin{bmatrix}I&\sqrt N\,S^{N+1}\\0&0\end{bmatrix},
   \qquad
   B_N=\begin{bmatrix}\sqrt{N+1}\,S^N&0\\0&0\end{bmatrix}.
 \]
 The inequality just proved gives $A_NA_N^*\succeq B_NB_N^*$.
 If a matrix multiplier $C$ satisfies $A_NM_C=B_N$, then
 \[
   C_{11}(z)+\sqrt N\,z_1^{N+1}C_{21}(z)
       =\sqrt{N+1}\,z_1^N.
 \]
 Thus the coefficient of $\zeta^N$ in
 $C_{11}(\zeta,0,\ldots,0)$ is $\sqrt{N+1}$. Cauchy's
 coefficient estimate and the pointwise multiplier norm bound
 imply
 \[
   \sqrt{N+1}
      \le\sup_{|\zeta|<1}|C_{11}(\zeta,0,\ldots,0)|
      \le\|M_C\|.
 \]
 Applying the preceding direct-sum construction with
 $2\times2$ blocks and with the common denominator
 $\|A_N\|$ for each pair produces
 $A,B\in\mathcal M(k_\alpha)\otw B(\ell^2)$ satisfying
 $AA^*\succeq BB^*$ for which there is no bounded multiplier
 solution of $AX=B$. This covers all the Bergman and Hardy
 ball spaces in the statement and completes the proof.
\end{proof}

\section{Examples and Questions}
 \label{sec:egandq}
 The multiplier algebra of a reproducing kernel Hilbert space can
 have the property that $AA^*\succeq BB^*$ implies the existence of
 a multiplier $C$ with $AC=B$, even though a contractive solution
 need not exist.

\begin{example}\rm
 \label{ex1}
 Let $k$ be the kernel on the unit disk given by
 \[
   k(z,w)=1+2\frac{z\overline w}{1-z\overline w}
          =\frac{1+z\overline w}{1-z\overline w},
 \]
 and let $s(z,w)=(1-z\overline w)^{-1}$ be the Szeg\"o kernel.
 Failure of the Douglas property is witnessed by a finite matrix
 pair.  Let $T=M_z$ on $H^2(k)$.
 The functions
 \[
   e_0=1,\qquad e_n=\sqrt{2}\,z^n\quad(n\ge1)
 \]
 form an orthonormal basis, and
 \[
   Te_0=\frac{1}{\sqrt{2}}e_1,
   \qquad Te_n=e_{n+1}\quad(n\ge1).
 \]
 Hence $I+2T^2T^{*2}-2TT^*$ is diagonal in this basis, with
 diagonal entries $1,0,0,1,1,\ldots$, and is positive.
 Set
 \[
   A=\begin{pmatrix}I&\sqrt{2}\,T^2\\0&0\end{pmatrix},
   \qquad
   B=\begin{pmatrix}\sqrt{2}\,T&0\\0&0\end{pmatrix}.
 \]
 Then $AA^*\succeq BB^*$.  If a multiplier matrix
 $C=(M_{c_{ij}})_{i,j=1}^2$ satisfies $AC=B$, then
 \[
   c_{11}(z)+\sqrt{2}\,z^2c_{21}(z)=\sqrt{2}\,z.
 \]
 In particular, $c_{11}(0)=0$ and $c_{11}'(0)=\sqrt{2}$.
 The pointwise bound $|c_{11}(z)|\le\|C\|$ and Schwarz's
 lemma imply $\|C\|\ge\sqrt{2}$.  Thus no contractive
 solution exists.

 We next show that every admissible factorization problem over
 this multiplier algebra has a solution of norm at most
 $\sqrt{2}$.  First, the inequalities
 \[
   s\preceq k\preceq2s
 \]
 hold in the order on positive semidefinite kernels.  Moreover,
 $k(z,w)=(1+z\overline w)s(z,w)$, and the kernel
 $1+z\overline w$ is positive semidefinite.  Since
 $1\in H^2(k)$, every multiplier on $H^2(k)$ has an analytic
 symbol.  For a bounded analytic operator-valued function
 $\Phi$, write $M_\Phi^k$ and $M_\Phi^s$ for the
 corresponding multiplication operators on the spaces with
 kernels $k$ and $s$, respectively.  If
 $\sup_{z\in\mathbb D}\|\Phi(z)\|\le1$, then
 \[
   (I-\Phi(z)\Phi(w)^*)s(z,w)
 \]
 is a positive semidefinite operator-valued kernel.
 Multiplying by $1+z\overline w$ shows that
 $\|M_\Phi^k\|\le1$.  Conversely, the kernel-vector formula
 in Lemma~\ref{lem:eigs-adjoints} gives
 $\|\Phi(z)\|\le\|M_\Phi^k\|$ for every $z$.
 It follows that the multiplier spaces for
 $k$ and $s$ agree as sets and, for all coefficient Hilbert
 spaces,
 \[
   \|M_\Phi^k\|=\|M_\Phi^s\|
      =\sup_{z\in\mathbb D}\|\Phi(z)\|.
 \]

 Now let $A,B\in\CM(k)\otw B(\ell^2)$ satisfy
 $AA^*\succeq BB^*$, and use $A(z),B(z)$ for their symbols.
 Put
 \[
   F_A(z,w)=A(z)A(w)^*,
   \qquad F_B(z,w)=B(z)B(w)^*.
 \]
 The operator inequality is equivalent to $F_Bk\preceq F_Ak$.
 Products of positive semidefinite operator-valued kernels
 with positive semidefinite scalar kernels are positive
 semidefinite.  Consequently,
 \[
   F_Bs\preceq F_Bk\preceq F_Ak\preceq2F_As.
 \]
 Leech's theorem~\cite{L}, applied on the Hardy space, yields
 a multiplier $C$ such that $AC=B$ and
 $\|M_C^s\|\le\sqrt{2}$.
 The equality of multiplier norms gives
 $\|M_C^k\|\le\sqrt{2}$.  Together with the finite matrix
 example above, this shows that the optimal uniform bound
 for solutions is $\sqrt{2}$.
\end{example}

 The example naturally leads to the following questions.

\begin{problem}\rm
 \label{prob:bounded-douglas}
 Say that a wot-closed unital algebra $\mathcal A$ has the
 {\it bounded Douglas property} if it satisfies the conditions
 of Definition~\ref{def:douglas} with the requirement
 $\|C\|\le1$ omitted.  In this case there is a constant
 $\gamma$, depending only on $\mathcal A$, such that every
 pair $A,B$ in that definition admits a solution $C$ with
 $AC=B$ and $\|C\|\le\gamma$.

 To see this, suppose no such uniform bound exists.  For each
 $n$ choose $A_n,B_n\in\mathcal A\otw B(\ell^2)$ such that
 $A_nA_n^*\succeq B_nB_n^*$ and every solution of
 $A_nC_n=B_n$ has norm greater than $n$.  Scaling each pair
 by a common nonzero scalar, arrange that
 $\|A_n\|\le1$, and hence $\|B_n\|\le1$.
 After identifying $\bigoplus_{n=1}^\infty\ell^2$ with
 $\ell^2$, the direct sums $A=\bigoplus_n A_n$ and
 $B=\bigoplus_n B_n$ belong to
 $\mathcal A\otw B(\ell^2)$ and satisfy $AA^*\succeq BB^*$.
 Any bounded solution $AC=B$ would have diagonal blocks
 satisfying $A_nC_{nn}=B_n$, and thus
 $\|C\|\ge\|C_{nn}\|>n$ for every $n$, a contradiction.

 Characterize those nice reproducing kernels $k$ for which
 $\CM(k)$ has the bounded Douglas property.
\end{problem}

\begin{problem}\rm
 \label{probI}
 Can the $B$ in Theorem~\ref{thm:main-aux} be chosen to be $I$?
\end{problem}

 See Trent~\cite{T} for the relevance of Problem~\ref{probI}
 to the corona problem for the bidisk.

 The following example shows that the hypothesis that $k$ is
 nice cannot simply be omitted from Theorem~\ref{thm:main}.

\begin{example}\rm
 \label{ex:not-nice}
 Let $\Omega=\mathbb C$ and $k(z,w)=\exp(z\overline w)$.
 In this case, $H^2(k)$ consists of the entire functions $f$
 such that
 \[
   \int_{\mathbb C}|f(z)|^2\exp(-|z|^2)\,dA(z)<\infty.
 \]
 Every multiplier is entire and satisfies
 $|\phi(z)|\le\|M_\phi\|$ for all $z\in\mathbb C$.
 Liouville's theorem therefore gives $\CM(k)=\mathbb C$,
 so $\CM(k)$ has the Douglas property.

 The kernel is not complete Pick.  At the base point $0$ the
 function required by Definition~\ref{def:NPk} is
 \[
   1-\exp(-z\overline w)
       =z\overline w-\frac12(z\overline w)^2+\cdots,
 \]
 whose coefficient of $(z\overline w)^2$ is negative.
 Explicitly, its Gram matrix at $r,-r$, for $r>0$, has the
 negative eigenvalue $2-\exp(r^2)-\exp(-r^2)$.
 Finally, $k$ is not nice: all scalar entries of the
 multipliers $p,q$ in Definition~\ref{def:nice} would be
 constant, making $p(z)p(w)^*-q(z)q(w)^*$ constant,
 whereas $1/k(z,w)=\exp(-z\overline w)$ is nonconstant.
\end{example}

\section{Finite matrices and countable amplifications}
\label{app:finite-matrices}

The Douglas property in Definition~\ref{def:douglas} implies its
finite-matrix version, by embedding finite matrices as corners of
$\mathcal A\otw B(\ell^2)$ and compressing a contractive solution.
The converse fails for general wot-closed unital operator algebras,
even on a separable Hilbert space.  We first construct the Hilbert
space and algebra $\mathcal{A}$ used in the counterexample.

\medskip
\subsection{Fock space and its multiplier algebra}
Let $\mathbb F_\infty^+$ be the set of finite words in the alphabet
$\mathbb N=\{1,2,\ldots\}$, including the empty word
$\varnothing$.  The {\it{full Fock space on countably many generators}} is
\[
  \mathcal F=\ell^2(\mathbb F_\infty^+).
\]
For a word $w$, let $e_w$ denote its standard basis vector.
Concatenation of words is written by juxtaposition.  The left and
right creation operators on $\mathcal F$ are defined by
\[
  L_j e_w=e_{jw},\qquad R_j e_w=e_{wj}\qquad(j\geq1).
\]
Each of these operators maps an orthonormal basis to an
orthonormal family and is therefore an isometry.  Left and right
creation commute: $L_iR_j=R_jL_i$ for all $i,j$.
We use the Fock-space multiplier algebra in its commutant form,
\[
  \mathcal L_\infty
   =\{a\in B(\mathcal F):aR_j=R_ja\text{ for every }j\geq1\}.
\]
Commutation with a fixed bounded operator is a wot-closed
condition.  Thus $\mathcal L_\infty$ is a wot-closed unital
algebra, and it contains every $L_j$.

For $a\in\mathcal L_\infty$, write
\[
  ae_\varnothing=\sum_{u\in\mathbb F_\infty^+}a_u e_u.
\]
Then $\sum_u|a_u|^2<\infty$, and commutation with the right
creation operators gives
\begin{equation}
\label{eq:app-fock-coefficients}
  ae_v=\sum_u a_u e_{uv},
  \qquad
  a^*e_v=\sum_{v=uw}\overline{a_u}\,e_w.
\end{equation}
The second sum is over the finitely many decompositions of the
word $v$.

\begin{lemma}
\label{lem:app-fock-factorization}
The algebra $\mathcal L_\infty$ has the finite-matrix
factorization property.  More precisely, let $r,s,t$ be positive
integers, and suppose
\[
  F_0\in M_{r,s}(\mathcal L_\infty),
  \qquad
  G_0\in M_{r,t}(\mathcal L_\infty)
\]
satisfy $F_0F_0^*\succeq G_0G_0^*$.  Then there exists
$C_0\in M_{s,t}(\mathcal L_\infty)$ such that
\[
  F_0C_0=G_0,
  \qquad
  \|C_0\|\leq1.
\]
\end{lemma}

\begin{proof}
Regard $F_0$ and $G_0$ as operators from
$\mathcal F^{\oplus s}$ and $\mathcal F^{\oplus t}$,
respectively, into $\mathcal F^{\oplus r}$.
The assumed inequality gives
\[
  \|G_0^*\xi\|^2\leq\|F_0^*\xi\|^2
  \qquad(\xi\in\mathcal F^{\oplus r}).
\]
Consequently, the rule
\[
  D(F_0^*\xi)=G_0^*\xi
\]
defines a well-defined contraction on $\operatorname{ran}F_0^*$.
It extends uniquely to a contraction
\[
  D:\mathcal K\longrightarrow\mathcal F^{\oplus t},
  \qquad
  \mathcal K=\overline{\operatorname{ran}F_0^*}
       \subseteq\mathcal F^{\oplus s}.
\]

For a positive integer $d$, write
$R_j^{(d)}=R_j\otimes I_{\mathbb C^d}$.
Since the entries of $F_0$ and $G_0$ commute with every $R_j$,
\[
  (R_j^{(s)})^*F_0^*
       =F_0^*(R_j^{(r)})^*,
  \qquad
  (R_j^{(t)})^*G_0^*
       =G_0^*(R_j^{(r)})^*.
\]
The first identity shows that $\mathcal K$ is invariant under
$(R_j^{(s)})^*$.  Together, the two identities imply
\[
  D\bigl((R_j^{(s)})^*|_{\mathcal K}\bigr)
       =(R_j^{(t)})^*D
  \qquad(j\geq1).
\]
Equivalently, if $P_{\mathcal K}$ denotes the orthogonal
projection onto $\mathcal K$, then
\[
  \bigl(P_{\mathcal K}R_j^{(s)}|_{\mathcal K}\bigr)D^*
       =D^*R_j^{(t)}
  \qquad(j\geq1).
\]

Apply the intertwining form of noncommutative commutant lifting
\cite{Pop06} to $D^*$.  It gives a contraction
\[
  C_0:\mathcal F^{\oplus t}\longrightarrow\mathcal F^{\oplus s}
\]
such that
\[
  P_{\mathcal K}C_0=D^*,
  \qquad
  C_0R_j^{(t)}=R_j^{(s)}C_0
  \quad(j\geq1).
\]
The intertwining identities show that every matrix entry of
$C_0$ commutes with every $R_j$.  Thus
$C_0\in M_{s,t}(\mathcal L_\infty)$.
Moreover, $C_0^*|_{\mathcal K}=D$, so
\[
  C_0^*F_0^*\xi
       =D(F_0^*\xi)
       =G_0^*\xi
  \qquad(\xi\in\mathcal F^{\oplus r}).
\]
Taking adjoints gives $F_0C_0=G_0$, as required.
\end{proof}

\medskip
\subsection{The algebra \texorpdfstring{$\mathcal A$}{A}}
We construct $\mathcal A$, beginning with the Hilbert space
on which it acts.

Consider the sequence of copies of $\mathcal F$ with the
isometric connecting maps
\[
  \mathcal F\xrightarrow{R_1}\mathcal F
  \xrightarrow{R_2}\mathcal F
  \xrightarrow{R_3}\cdots .
\]
For $m>n$, set
\[
  V_{m,n}=R_{m-1}\cdots R_n,\qquad V_{n,n}=I_{\mathcal F}.
\]
The maps $V_{m,n}$ are isometries and satisfy
$V_{\ell,m}V_{m,n}=V_{\ell,n}$ whenever $\ell\geq m\geq n$.

Let $H_{\mathrm{alg}}$ be the algebraic direct limit of this
system.  Its elements are equivalence classes $[n,\xi]$ of pairs
with $n\geq1$ and $\xi\in\mathcal F$, where
\[
 [n,\xi]=[m,\eta]
 \quad\Longleftrightarrow\quad
 V_{\ell,n}\xi=V_{\ell,m}\eta
 \text{ for some }\ell\geq n,m.
\]
Vector operations are defined by transporting representatives to
a common stage; in particular,
\[
 [n,\xi]+[m,\eta]
   =[\ell,V_{\ell,n}\xi+V_{\ell,m}\eta]
   \qquad(\ell\geq n,m).
\]
Every finite linear combination can therefore be represented by
a single class $[n,\xi]$.

Define an inner product on $H_{\mathrm{alg}}$ by
\begin{equation}
\label{eq:app-limit-inner-product}
 \big\langle[n,\xi],[m,\eta]\big\rangle
   =\big\langle V_{\ell,n}\xi,V_{\ell,m}\eta\big\rangle_{\mathcal F},
   \qquad \ell\geq n,m.
\end{equation}
Passing to a later common stage preserves this expression, since
the connecting maps are isometries.  The same observation shows
that it is independent of the representatives.  Moreover,
\[
  \|[n,\xi]\|=\|\xi\|_{\mathcal F},
\]
so the inner product is positive definite.  Define $H$ to be the
Hilbert-space completion of this inner product space:
\[
  H=\overline{H_{\mathrm{alg}}}^{\,\|\cdot\|}.
\]
We identify $H_{\mathrm{alg}}$ with its canonical dense image in $H$.

The maps
\[
  J_n:\mathcal F\longrightarrow H,\qquad J_n\xi=[n,\xi],
\]
are isometries.  The defining identifications give
\begin{equation}
\label{eq:app-inductive-limit}
  J_n=J_{n+1}R_n.
\end{equation}
Each range $H_n=J_n\mathcal F$ is closed, because $\mathcal F$
is complete, and
\[
  H_n\subseteq H_{n+1},\qquad
  H_{\mathrm{alg}}=\bigcup_{n=1}^\infty H_n,\qquad
  H=\overline{\bigcup_{n=1}^\infty H_n}.
\]
The union of countable dense subsets of the $H_n$ is dense in
$H$, so $H$ is separable.  The orthogonal projections
$E_n=J_nJ_n^*$ onto $H_n$ converge strongly to $I_H$:
they eventually fix each vector in $H_{\mathrm{alg}}$, and
their norms are at most one.

\medskip
For $a\in\mathcal L_\infty$, define an operator on
$H_{\mathrm{alg}}$ by
\[
  \pi_0(a)[n,\xi]=[n,a\xi].
\]
This is well defined because $a$ commutes with each $R_j$ and
hence with every $V_{m,n}$.  For example,
\[
 [m,aV_{m,n}\xi]=[m,V_{m,n}a\xi]=[n,a\xi]
 \qquad(m\geq n).
\]
It is linear, and the norm identity above gives
\[
  \|\pi_0(a)[n,\xi]\|
    =\|a\xi\|
    \leq\|a\|\,\|[n,\xi]\|.
\]
Since every vector of $H_{\mathrm{alg}}$ has such a
representative, this is a bound on all of $H_{\mathrm{alg}}$.
Consequently, $\pi_0(a)$ extends uniquely to a bounded operator
$\pi(a)$ on $H$, with $\|\pi(a)\|\leq\|a\|$, and
\begin{equation}
\label{eq:app-representation}
  \pi(a)J_n=J_na\qquad(n\geq1).
\end{equation}
In particular, every $H_n$ is invariant under $\pi(a)$.
Restriction to $H_n$ gives $\|\pi(a)\|\geq\|a\|$, and hence
\[
  \|\pi(a)\|=\|a\|.
\]
The identities
\[
  \pi(a)\pi(b)J_n=J_nab=\pi(ab)J_n
\]
and density show that $\pi$ is multiplicative.  Linearity and
$\pi(I_{\mathcal F})=I_H$ follow in the same way, so $\pi$ is
a unital isometric homomorphism.

The same argument applies at every matrix level.  For a positive
integer $d$, let $J_n^{(d)}$ be the direct sum of $d$ copies of
$J_n$.  If $T=[a_{ij}]\in M_d(\mathcal L_\infty)$, then
\[
  [\pi(a_{ij})]J_n^{(d)}=J_n^{(d)}T.
\]
The subspaces $H_n^{\oplus d}$ increase and have dense union in
$H^{\oplus d}$.  The restriction of $[\pi(a_{ij})]$ to each
of these subspaces has norm $\|T\|$, so
\[
  \|[\pi(a_{ij})]\|=\|[a_{ij}]\|.
\]
Thus $\pi$ is completely isometric.  Set
\[
  \mathcal A=\pi(\mathcal L_\infty)\subseteq B(H).
\]

\begin{lemma}
\label{lem:app-wot-closed}
The algebra $\mathcal A$ is a wot-closed unital subalgebra of $B(H)$.
Moreover, for each $a\in\mathcal L_\infty$ and each word
$v\in\mathbb F_\infty^+$,
\[
  \|\pi(a)^*J_ne_v-J_na^*e_v\|\longrightarrow0
  \qquad(n\to\infty).
\]
\end{lemma}

\begin{proof}
Unitality and closure under multiplication have already been
proved.  We first establish wot-closedness.  Suppose a net
$\pi(a_\alpha)$, with $a_\alpha\in\mathcal L_\infty$,
converges wot to $T\in B(H)$.  Each $H_n$ is invariant under
every $\pi(a_\alpha)$ and is closed.  For $h\in H_n$ and
$g\in H_n^\perp$, therefore,
\[
  \langle Th,g\rangle
    =\lim_\alpha\langle\pi(a_\alpha)h,g\rangle=0.
\]
Hence $T$ also leaves every $H_n$ invariant.

By~\eqref{eq:app-representation},
\[
  J_n^*\pi(a_\alpha)J_n=a_\alpha
  \qquad\text{for every }n\text{ and }\alpha.
\]
Taking $n=1$ shows that $a_\alpha$ converges wot on
$\mathcal F$ to $a=J_1^*TJ_1$.  Since $\mathcal L_\infty$
is wot-closed, $a\in\mathcal L_\infty$.
For arbitrary $n$, taking the limit in the same identity gives
\[
  J_n^*TJ_n=a.
\]
The invariance of $H_n$ now implies
\[
  TJ_n=J_n(J_n^*TJ_n)=J_na=\pi(a)J_n.
\]
Thus $T$ agrees with $\pi(a)$ on every $H_n$.  By density and
boundedness, $T=\pi(a)\in\mathcal A$.

To prove the moreover statement, fix $a\in\mathcal L_\infty$
and $v\in\mathbb F_\infty^+$.
Equation~\eqref{eq:app-inductive-limit} gives
\[
  J_ne_v=J_m e_{v\,n(n+1)\cdots(m-1)}\qquad(m>n).
\]
Taking adjoints in~\eqref{eq:app-representation} gives
$J_m^*\pi(a)^*=a^*J_m^*$.  Consequently, for $m>n$,
\[
\begin{split}
  E_m\pi(a)^*J_ne_v
    &=J_ma^*e_{v\,n(n+1)\cdots(m-1)}\\
    &=J_na^*e_v+
      \sum_{k=n}^{m-1}
        \overline{a_{v\,n(n+1)\cdots k}}\,
        J_{k+1}e_\varnothing.
\end{split}
\]
Here the second equality follows
from~\eqref{eq:app-fock-coefficients}.
Since $E_m\to I_H$ strongly, letting $m\to\infty$ yields
\begin{equation}
\label{eq:app-adjoint-tail}
  \pi(a)^*J_ne_v
   =J_na^*e_v+
     \sum_{k=n}^{\infty}
       \overline{a_{v\,n(n+1)\cdots k}}\,
       J_{k+1}e_\varnothing.
\end{equation}
For $k\geq n$, the vector $J_{k+1}e_\varnothing$ belongs to
$H_{k+1}\ominus H_k$.  Hence these vectors are mutually orthogonal
and orthogonal to $H_n$.
It follows that
\begin{equation}
\label{eq:app-tail-vanishes}
\begin{split}
  \|\pi(a)^*J_ne_v-J_na^*e_v\|^2
    &=\sum_{k=n}^{\infty}|a_{v\,n(n+1)\cdots k}|^2\\
    &\leq
       \sum_{\substack{u\in\mathbb F_\infty^+\\
                      u\text{ begins with }vn}}
            |a_u|^2.
\end{split}
\end{equation}
The sets of words beginning with $vn$ are
pairwise disjoint as $n$ varies, and the Fourier coefficients of
$a$ are square summable.  Hence the sum on the right side of
equation~\eqref{eq:app-tail-vanishes} converges to $0$ as
$n\to\infty$.
\end{proof}

\subsection{The counterexample}
We now show that the algebra $\mathcal A$ distinguishes finite-matrix
factorization from the Douglas property.

\begin{theorem}
\label{thm:finite-not-infinite}
The algebra $\mathcal A=\pi(\mathcal L_\infty)$ constructed above
on the separable Hilbert space $H$ has the following properties.
\begin{enumerate}[(a)] \itemsep=6pt 
\item \label{i:ce:a}
 For all positive integers $r,s,t$, if
  $F\in M_{r,s}(\mathcal A)$ and $G\in M_{r,t}(\mathcal A)$ satisfy
  $FF^*\succeq GG^*$, then there exists
  $C\in M_{s,t}(\mathcal A)$ such that $FC=G$ and $\|C\|\leq 1$.
\item \label{i:ce:b}
There exist $A,B\in\mathcal A\otw B(\ell^2)$ with
  $AA^*=BB^*$ for which $AX=B$ has no solution
  $X\in\mathcal A\otw B(\ell^2)$, regardless of its norm.
\end{enumerate}
\end{theorem}

\begin{proof}

\medskip

Let $F=\pi(F_0)\in M_{r,s}(\mathcal A)$ and
$G=\pi(G_0)\in M_{r,t}(\mathcal A)$ satisfy $FF^*\succeq GG^*$,
where $\pi$ is applied entrywise.
Write $J_n^{(r)}$ for the direct sum of $r$ copies of $J_n$.
For every vector $\xi\in\mathcal F^{\oplus r}$ whose components
are finite linear combinations of basis vectors, the moreover
assertion of Lemma~\ref{lem:app-wot-closed} gives
\[
  \|F^*J_n^{(r)}\xi-J_n^{(s)}F_0^*\xi\|\longrightarrow0,
  \qquad
  \|G^*J_n^{(r)}\xi-J_n^{(t)}G_0^*\xi\|\longrightarrow0.
\]
Consequently,
\[
  \|F_0^*\xi\|^2
    =\lim_n\|F^*J_n^{(r)}\xi\|^2
    \geq\lim_n\|G^*J_n^{(r)}\xi\|^2
    =\|G_0^*\xi\|^2.
\]
Density implies $F_0F_0^*\succeq G_0G_0^*$.
By Lemma~\ref{lem:app-fock-factorization}, there exists
$C_0\in M_{s,t}(\mathcal L_\infty)$ such that
\[
  F_0C_0=G_0,
  \qquad
  \|C_0\|\leq1.
\]
Set $C=\pi(C_0)$, with $\pi$ applied entrywise.
Then $C\in M_{s,t}(\mathcal A)$ and $FC=G$.
Since $\pi$ is completely isometric,
\[
  \|C\|=\|C_0\|\leq1.
\]
This proves item~\ref{i:ce:a}.

\medskip

\medskip
Let $S_j=\pi(L_j)$.  By the definition of $\pi$,
\[
  S_j[n,\zeta]=[n,L_j\zeta]
  \qquad(n,j\geq1,\ \zeta\in\mathcal F).
\]
For $\xi,\eta\in\mathcal F$, the inner product on the
inductive limit gives
\[
\begin{split}
  \big\langle S_i[n,\xi],S_j[n,\eta]\big\rangle_H
    &=\langle L_i\xi,L_j\eta\rangle_{\mathcal F}\\
    &=\delta_{ij}\langle\xi,\eta\rangle_{\mathcal F}\\
    &=\delta_{ij}
        \big\langle[n,\xi],[n,\eta]\big\rangle_H.
\end{split}
\]
Any two vectors in $H_{\mathrm{alg}}$ can be represented at a
common stage, say $n.$   Since $H_{\mathrm{alg}}$ is dense in $H$, these
identities imply
\[
  S_i^*S_j=\delta_{ij}I_H.
\]
Thus the $S_j$ are isometries with mutually orthogonal ranges.

We next show that their range projections sum strongly to the
identity.  Set
\[
  \mathcal D
    =\operatorname{span}\{
        [n,e_w]:n\geq1,\ w\in\mathbb F_\infty^+
      \}.
\]
For each fixed $n$, the vectors $[n,e_w]$ span a dense subspace
of $H_n$, and the union of the $H_n$ is dense in $H$.
Consequently, $\mathcal D$ is dense in $H$.

Every vector $[n,e_w]$ belongs to the range of some $S_j$.
Indeed, if $w$ is nonempty, write $w=jv$, where $j$ is its
first letter.  Then
\[
  [n,e_w]=[n,L_je_v]=S_j[n,e_v].
\]
For the empty word, the defining identification between
successive stages gives
\[
\begin{split}
  [n,e_\varnothing]
    &=[n+1,R_ne_\varnothing]\\
    &=[n+1,e_n]\\
    &=S_n[n+1,e_\varnothing].
\end{split}
\]

For $N\geq1$, let
\[
  P_N=\sum_{j=1}^N S_jS_j^*.
\]
Since the ranges of the $S_j$ are mutually orthogonal,
$P_N$ is the orthogonal projection onto
$\bigoplus_{j=1}^N S_jH$.  In particular,
$\|I_H-P_N\|\leq1$.
The preceding identities show that, for every $d\in\mathcal D$,
there is an $N_0$ such that
\[
  P_Nd=d\qquad(N\geq N_0).
\]
Here only finitely many ranges are needed, since $d$ is a
finite linear combination of the vectors $[n,e_w]$.

Now let $h\in H$ and $\varepsilon>0$.
Choose $d\in\mathcal D$ with $\|h-d\|<\varepsilon$.
For all sufficiently large $N$, we have
\[
\begin{split}
  \|(I_H-P_N)h\|
    &=\|(I_H-P_N)(h-d)\|\\
    &\leq\|h-d\|\\
    &<\varepsilon.
\end{split}
\]
Thus $P_Nh\to h$ for every $h\in H$, proving
\[
  \sum_{j=1}^{\infty}S_jS_j^*=I_H
\]
in the strong operator topology.

Identify $H\otimes\ell^2$ with the Hilbert space of
square-summable sequences $(h_j)_{j\geq1}$ in $H$, and define
\[
  U:H\otimes\ell^2\longrightarrow H,
  \qquad
  U(h_1,h_2,\ldots)=\sum_{j=1}^{\infty}S_jh_j.
\]
The series converges in norm by orthogonality, and
\[
  \|U(h_1,h_2,\ldots)\|^2
     =\sum_{j=1}^{\infty}\|h_j\|^2.
\]
Hence $U$ is an isometry and has closed range.
Its range contains every $S_jH$, whose closed linear span is
$H$ by the strong convergence proved above.  Therefore $U$ is
onto and is unitary.  In row-operator notation,
\[
  U=[\,S_1\ S_2\ S_3\ \cdots\,].
\]

On the other hand, $S_j^*\notin\mathcal A$ for every $j$.
To see this, first observe that every member of $\mathcal A$
leaves
\[
  H_j=\{[j,\zeta]:\zeta\in\mathcal F\}
\]
invariant, since
\[
  \pi(a)[j,\zeta]=[j,a\zeta]
  \qquad(a\in\mathcal L_\infty).
\]
The identity for the empty word, with $n=j$, gives
\[
  [j,e_\varnothing]=S_j[j+1,e_\varnothing].
\]
Since $S_j$ is an isometry, it follows that
\[
  S_j^*[j,e_\varnothing]=[j+1,e_\varnothing].
\]
For every $\zeta\in\mathcal F$, however,
\[
\begin{split}
  \big\langle[j+1,e_\varnothing],[j,\zeta]\big\rangle_H
    &=\big\langle[j+1,e_\varnothing],
                   [j+1,R_j\zeta]\big\rangle_H\\
    &=\langle e_\varnothing,R_j\zeta\rangle_{\mathcal F}\\
    &=0.
\end{split}
\]
Thus $[j+1,e_\varnothing]$ is a unit vector in
$H_{j+1}\ominus H_j$.  Consequently, $S_j^*$ does not leave
$H_j$ invariant and cannot belong to $\mathcal A$.

Let $\iota:H\to H\otimes\ell^2$ be the inclusion into the first
coordinate, and set
\[
  A=\iota U,
  \qquad
  B=\iota\iota^*=I_H\otimes E_{11},
\]
where $E_{ij}$ denotes the standard matrix unit on $\ell^2$.
Thus $A$ has first row $[S_1,S_2,\ldots]$ and all other rows
zero, while $B$ has $I_H$ in its $(1,1)$ entry and zeros
elsewhere.

Let $Q_N$ be the projection of $\ell^2$ onto its first $N$
coordinates.  Then
\[
  A(I_H\otimes Q_N)
     =\sum_{j=1}^N S_j\otimes E_{1j}
     \in\mathcal A\otimes B(\ell^2).
\]
Since $I_H\otimes Q_N\to I_{H\otimes\ell^2}$ strongly,
these finite row truncations converge strongly to $A$.
Also, $B=I_H\otimes E_{11}$ belongs to the algebraic tensor
product.  Therefore
\[
  A,B\in\mathcal A\otw B(\ell^2).
\]
The unitarity of $U$ gives
\[
  AA^*=\iota UU^*\iota^*
       =\iota\iota^*
       =BB^*.
\]

Suppose that $AX=B$ for some
$X\in\mathcal A\otw B(\ell^2)$.
Multiplying this equation on the left by $\iota^*$ and on
the right by $\iota$ yields
\[
  UX\iota=I_H.
\]
Since $U$ is unitary, we obtain
\[
  X\iota=U^*.
\]
The $(j,1)$ matrix entry of $X$ is therefore $S_j^*$.
Taking a matrix entry is wot-continuous, so every matrix entry
of an element of $\mathcal A\otw B(\ell^2)$ belongs to the
wot-closed algebra $\mathcal A$.
This contradicts $S_j^*\notin\mathcal A$ and proves item~\ref{i:ce:b}.
\end{proof}

\begin{remark}\rm
\label{rem:app-cstar-topologies}
The $C^*$-algebras associated with this construction should be
distinguished from the usual finite-generator Cuntz--Toeplitz
algebra.  Let $\mathcal F_d=\ell^2(\mathbb F_d^+)$, where
$2\leq d<\infty$, and let $L_1^{(d)},\ldots,L_d^{(d)}$ be the
left creation operators on $\mathcal F_d$.  Then
\[
  P_\varnothing^{(d)}
    =I_{\mathcal F_d}
       -\sum_{j=1}^d L_j^{(d)}(L_j^{(d)})^*
\]
is the rank-one vacuum projection.  It belongs to
\[
  \mathcal E_d=C^*(L_1^{(d)},\ldots,L_d^{(d)})
\]
and generates the ideal $\mathcal K(\mathcal F_d)$ of compact
operators.  In particular,
$\mathcal K(\mathcal F_d)\subseteq\mathcal E_d$.
The quotient $\mathcal E_d/\mathcal K(\mathcal F_d)$ is the
Cuntz algebra $\mathcal O_d$; see~\cite{DKP}.
This comparison uses the $d$-letter Fock space: for the first
$d$ creators on the countably generated Fock space
$\mathcal F$ used above, the projection
\[
  I_{\mathcal F}-\sum_{j=1}^d L_jL_j^*
\]
has infinite rank.

In the present setting, put
\[
  \mathcal C=C^*(\mathcal L_\infty)\subseteq B(\mathcal F).
\]
Here $\mathcal C$ is generated by the entire wot-closed
algebra $\mathcal L_\infty$.  In fact,
\[
  \mathcal C\cap\mathcal K(\mathcal F)=\{0\}.
\]
To verify this claim, write
$ae_\varnothing=\sum_w a_w e_w$ for
$a\in\mathcal L_\infty$.  The Fourier-coefficient formula
gives
\[
\begin{aligned}
  (a^*R_n-R_na^*)e_v
     &=\overline{a_{vn}}\,e_\varnothing,\\
  \|a^*R_n-R_na^*\|^2
     &=\sum_v|a_{vn}|^2\longrightarrow0.
\end{aligned}
\]
Indeed, the coefficients are square summable, and the sets
of words ending in $n$ are pairwise disjoint.
Since $aR_n=R_na$, passage to $*$-polynomials and then norm
approximation gives
\[
  \|TR_n-R_nT\|\longrightarrow0
  \qquad(T\in\mathcal C).
\]
The right creation operators have mutually orthogonal
ranges, so $\|KR_n\|\to0$ for every compact operator $K$.
Thus, if $T\in\mathcal C$ is compact, then
\[
  \|T\|=\|R_nT\|
     \leq \|TR_n\|+\|TR_n-R_nT\|
     \longrightarrow0.
\]
Hence $\mathcal C\cap\mathcal K(\mathcal F)=\{0\}$.

The representation $\pi$ extends to a canonical
$*$-isomorphism
\[
  \sigma:C^*(\mathcal L_\infty)
          \longrightarrow C^*(\mathcal A),
  \qquad
  \sigma(a)=\pi(a)\quad(a\in\mathcal L_\infty).
\]
Indeed, with $E_n=J_nJ_n^*$, the maps
$T\mapsto J_nTJ_n^*$ are nonunital $*$-homomorphisms, and
\[
  J_naJ_n^*=\pi(a)E_n
       \longrightarrow\pi(a),
  \qquad
  J_na^*J_n^*=E_n\pi(a)^*
       \longrightarrow\pi(a)^*
\]
strongly.  Products of uniformly bounded strongly convergent
operators converge strongly.  Thus passage to
$*$-polynomials and norm approximation defines a unital
$*$-homomorphism
\[
  \sigma(T)=\operatorname{sot}\!-\!\lim_{n\to\infty}
                  J_nTJ_n^*
  \qquad(T\in\mathcal C).
\]
Moreover, \eqref{eq:app-tail-vanishes}, extended by density,
and the identity $\pi(a)J_n=J_na$ imply
\[
  \|\sigma(T)J_n\xi-J_nT\xi\|\longrightarrow0
  \qquad(T\in\mathcal C,\ \xi\in\mathcal F),
\]
first for $*$-polynomials and then for all $T\in\mathcal C$.
Consequently $\|\sigma(T)\|\geq\|T\|$.
The reverse inequality follows from contractivity, so
$\sigma$ is isometric.  Its range is closed and contains
$\mathcal A$, and hence is all of $C^*(\mathcal A)$.

Nevertheless, $\sigma$ is neither sot nor wot continuous,
even on the unit ball as we now verify. Put
\[
  w_N=1\,2\,\cdots N,
  \qquad
  Q_N=L_{w_N}L_{w_N}^*,
  \qquad
  L_{w_N}=L_1L_2\cdots L_N.
\]
The range of $Q_N$ is contained in the closed span of the
words of length at least $N$, so $Q_N\to0$ strongly on
$\mathcal F$.  On the other hand, the unit vector
$h=[1,e_\varnothing]$ satisfies
\[
  h=[N+1,e_{w_N}]
    =S_{w_N}[N+1,e_\varnothing],
  \qquad
  S_{w_N}=S_1S_2\cdots S_N.
\]
Therefore
\[
  \sigma(Q_N)h=S_{w_N}S_{w_N}^*h=h
  \qquad(N\geq1).
\]
In particular,
$\langle\sigma(Q_N)h,h\rangle=1$ for every $N$, so
$\sigma(Q_N)$ does not converge weakly to zero.

The inverse isomorphism is likewise neither sot nor wot
continuous.  Indeed,
\[
  \sum_{j=1}^N S_jS_j^*
      \longrightarrow I_H
      \qquad\text{strongly},
\]
whereas
\[
  \sigma^{-1}\left(\sum_{j=1}^N S_jS_j^*\right)
     =\sum_{j=1}^N L_jL_j^*
     \longrightarrow I_{\mathcal F}-P_\varnothing
     \qquad\text{strongly}.
\]
These preimages do not converge weakly to
$\sigma^{-1}(I_H)=I_{\mathcal F}$, as is seen by evaluating
on $e_\varnothing$.  Thus the canonical $*$-isomorphism
preserves all finite $*$-polynomial relations and norm
limits, but does not preserve the strong or weak operator
topologies of these concrete representations.

A precise interpretation of the construction of $\mathcal A$
in terms of quotients by compact operators goes as follows.
Let
\[
  \mathcal K=\mathcal K(\mathcal F),
  \qquad
  q:B(\mathcal F)\longrightarrow B(\mathcal F)/\mathcal K
\]
be the Calkin quotient map.  Since
$\mathcal C\cap\mathcal K=\{0\}$, the restriction
$q|_{\mathcal C}$ is a faithful $*$-homomorphism and hence
is completely isometric.  Consequently,
\[
  \rho:q(\mathcal C)\longrightarrow C^*(\mathcal A),
  \qquad
  \rho(q(T))=\sigma(T),
\]
is a $*$-isomorphism satisfying
\[
  \rho(q(a))=\pi(a)
  \qquad(a\in\mathcal L_\infty).
\]
In particular, there are canonical completely isometric
algebra isomorphisms
\[
  \mathcal A
    \cong_{\mathrm{c.i.}}q(\mathcal L_\infty)
    \cong_{\mathrm{c.i.}}
       (\mathcal L_\infty+\mathcal K)/\mathcal K,
\]
and, at the $C^*$-algebra level, canonical $*$-isomorphisms
\[
  C^*(\mathcal A)
    \cong q(\mathcal C)
    \cong (\mathcal C+\mathcal K)/\mathcal K
    =C^*(\mathcal L_\infty,\mathcal K)/\mathcal K.
\]
The sums $\mathcal L_\infty+\mathcal K$ and
$\mathcal C+\mathcal K$ are norm closed: the spaces
$q(\mathcal L_\infty)$ and $q(\mathcal C)$ are norm closed,
and
\[
  \mathcal L_\infty+\mathcal K
     =q^{-1}\bigl(q(\mathcal L_\infty)\bigr),
  \qquad
  \mathcal C+\mathcal K=q^{-1}\bigl(q(\mathcal C)\bigr).
\]
Equivalently,
\[
  \widetilde{\sigma}:\mathcal C+\mathcal K
        \longrightarrow C^*(\mathcal A),
  \qquad
  \widetilde{\sigma}(T+K)=\sigma(T)
  \quad(T\in\mathcal C,\ K\in\mathcal K),
\]
is a surjective $*$-homomorphism with kernel exactly
$\mathcal K$.
Since $\widetilde{\sigma}|_{\mathcal C}=\sigma$,
the map $\widetilde{\sigma}$ is neither sot nor wot
continuous, even on the unit ball.
The vacuum projection $P_\varnothing$ is not in
$\mathcal C$, because it is a nonzero compact operator
and $\mathcal C\cap\mathcal K=\{0\}$.
It belongs to $\mathcal C+\mathcal K$ and satisfies
$\widetilde{\sigma}(P_\varnothing)=0$.
Thus the compact-operator quotients above are formed
after adjoining $\mathcal K$; passing to the Calkin
image does not identify distinct elements of
$\mathcal L_\infty$, or even of $\mathcal C$.

Finally, although $q(P_\varnothing)=0$, the strong identity
\[
  \sum_{j=1}^{\infty}L_jL_j^*
     =I_{\mathcal F}-P_\varnothing
\]
does not give a norm-convergent sum in the Calkin algebra.
Indeed,
$D_N=I_{\mathcal F}-\sum_{j=1}^N L_jL_j^*$
is an infinite-rank projection, so
\[
  \left\|q(I_{\mathcal F})
       -\sum_{j=1}^Nq(L_j)q(L_j)^*\right\|
     =\|q(D_N)\|=1.
\]
Strong convergence must be considered in a specified
concrete representation.  For the representation
$\rho:q(\mathcal C)\to B(H)$ constructed above,
\[
  \rho\left(\sum_{j=1}^Nq(L_j)q(L_j)^*\right)
    =\sum_{j=1}^N S_jS_j^*
    \longrightarrow I_H
    \qquad\text{strongly}.
\]
In fact, this particular strong limit is respected by
$\widetilde{\sigma}$, since
$\widetilde{\sigma}(I_{\mathcal F}-P_\varnothing)=I_H$.
\end{remark}
  
\begin{remark}\rm
\label{rem:finite-commutative}
The algebra in Theorem~\ref{thm:finite-not-infinite} is
noncommutative.  The construction does not resolve the more
restricted question whether the finite-matrix factorization
property implies the Douglas property for a commutative wot-closed
unital algebra.  In particular, it does not settle that implication
for arbitrary scalar multiplier algebras $\CM(k)$.
Section~\ref{sec:finite-countable} establishes affirmative results
for complete Pick kernels and for nice kernels admitting a
representation with finitely many components in $p$.
\end{remark}

\section{Finite and countable factorizations}
\label{sec:finite-countable}

This section identifies two settings in which finite-matrix
factorization and the Douglas property agree.  For complete Pick
kernels, the full Douglas property follows directly from the
operator-valued factorization theorem of Ball, Trent, and Vinnikov
\cite{BTV}.  For nice kernels whose defining row $p$ can be chosen
with finitely many components, the finite-matrix property already
forces the kernel to be complete Pick.

For clarity, say that a wot-closed unital algebra $\mathcal A$ has
the \emph{finite-matrix factorization property} if, for every choice
of positive integers $r,s,t$ and every
$F\in M_{r,s}(\mathcal A)$ and $G\in M_{r,t}(\mathcal A)$,
the inequality $FF^*\succeq GG^*$ implies the existence of
$C\in M_{s,t}(\mathcal A)$ with $FC=G$ and $\|C\|\leq1$.
By padding with zero rows and columns, this is equivalent to the
corresponding condition for finite square matrices.  The term
\emph{Douglas property} continues to mean
Definition~\ref{def:douglas}.

\begin{theorem}
\label{thm:cp-douglas-countable}
Suppose $k$ is a complete Pick kernel, $H^2(k)$ is separable,
and $k(z,z)>0$ for every $z\in\Omega$.  Then $\CM(k)$ has
the Douglas property.  In particular, both its finite-matrix and
countably amplified factorization properties hold.
\end{theorem}

\begin{proof}
First suppose that $k(z,w)\ne0$ for all $z,w\in\Omega$.
Let $A,B\in\CM(k)\otw B(\ell^2)$ satisfy
$AA^*\succeq BB^*$.  By Lemma~\ref{lem:op-mults}, these
operators have $B(\ell^2)$-valued multiplier symbols, denoted
by $A(z)$ and $B(z)$ as well.  Evaluating on finite sums of
kernel vectors shows that the operator-valued kernel
\[
 K(z,w)=k(z,w)\bigl[A(z)A(w)^*-B(z)B(w)^*\bigr]
\]
is positive semi-definite.  The complete Pick version of Leech's
factorization theorem in~\cite{BTV} supplies
$\Phi\in\cM{\ell^2}{\ell^2}$ such that
\[
 A(z)\Phi(z)=B(z)\quad(z\in\Omega),
 \qquad \|M_\Phi\|\leq1.
\]
The coefficient Hilbert spaces in that theorem may be infinite
dimensional.  Lemma~\ref{lem:op-mults} therefore identifies
$C=M_\Phi$ as an element of $\CM(k)\otw B(\ell^2)$, and
$AC=B$.  No finite-matrix factorization hypothesis is needed.

The normalization convention for complete Pick kernels does not
affect this argument.  Indeed, for a fixed $\omega\in\Omega$, put
\[
 \delta(z)=\frac{k(z,\omega)}{\sqrt{k(\omega,\omega)}},
 \qquad
 \widehat k(z,w)=\frac{k(z,w)}{\delta(z)\overline{\delta(w)}}.
\]
Then $\widehat k(z,\omega)=1$, and multiplication by $\delta$
is a unitary from $H^2(\widehat k)$ onto $H^2(k)$.  Its
amplifications intertwine the corresponding multipliers and
preserve their norms and operator inequalities.

For a kernel that may vanish off the diagonal, equation~\eqref{eq:np}
with base point $y$ gives
\[
 k(x,z)=0\quad\Longrightarrow\quad k(x,y)k(y,z)=0.
\]
Define $x\sim y$ to mean $k(x,y)\ne0$.  The implication above
and positivity of the diagonal show that this is an equivalence
relation.
Write its classes as $\Omega_\lambda$, and let $k_\lambda$
be the restriction of $k$ to
$\Omega_\lambda\times\Omega_\lambda$.  Each $k_\lambda$ is
a nonvanishing complete Pick kernel, and
\[
 H^2(k)=\bigoplus_\lambda H^2(k_\lambda).
\]
There are at most countably many nonzero summands, since $H^2(k)$
is separable.  Multiplication operators, including their coefficient
amplifications, act componentwise in this decomposition.  Restrict
$A$ and $B$ to each summand and apply the nonvanishing case.
The resulting multiplier solutions $C_\lambda$ all have norm at
most one.  Their direct sum is a contractive multiplier on
$H^2(k)\otimes\ell^2$ solving $AC=B$, and belongs to
$\CM(k)\otw B(\ell^2)$ by Lemma~\ref{lem:op-mults}.
\end{proof}

The following elementary compactness argument will be used for
nice kernels.  Its crucial feature is that the left row is fixed
and finite while the right row is truncated.

\begin{lemma}
\label{lem:finite-left-row}
Suppose a wot-closed unital algebra $\mathcal A\subseteq B(H)$
has the finite-matrix factorization property.  Let
$F=[f_1,\ldots,f_r]\in M_{1,r}(\mathcal A)$, and let
$G=[g_1,g_2,\ldots]:H\otimes\ell^2\to H$ be a bounded
row operator with $g_j\in\mathcal A$.  If
$FF^*\succeq GG^*$, then there is a contractive operator
$C:H\otimes\ell^2\to H\otimes\mathbb C^r$ whose matrix
entries belong to $\mathcal A$ and for which $FC=G$.
\end{lemma}

\begin{proof}
Let $P_n$ be the projection of $\ell^2$ onto its first $n$
coordinates, and write $G_n=[g_1,\ldots,g_n]$.  Then
\[
 FF^*\succeq GG^*\succeq G_nG_n^*.
\]
The finite-matrix property gives
$C_n\in M_{r,n}(\mathcal A)$ such that
$FC_n=G_n$ and $\|C_n\|\leq1$.  Pad $C_n$ with zero
columns to obtain
$\widehat C_n:H\otimes\ell^2\to H\otimes\mathbb C^r$.
The unit ball of this operator space is wot-compact, so a subnet
of $\widehat C_n$ converges wot to a contraction $C$.
Every matrix entry of $C$ belongs to $\mathcal A$ by
wot-closedness.  Moreover,
\[
 F\widehat C_n=G(I_H\otimes P_n)\longrightarrow G
\]
strongly.  Since multiplication on the left by the fixed bounded
operator $F$ is wot-continuous, it follows that $FC=G$.
\end{proof}

\begin{theorem}
\label{thm:nice-finite-p}
Suppose $k$ is nice and that a representation in
Definition~\ref{def:nice} can be chosen with $p$ having only
finitely many nonzero components.  Then the following are
equivalent:
\begin{enumerate}[(i)]\itemsep=6pt
\item \label{i:nice:i} $\CM(k)$ has the finite-matrix factorization property;
\item \label{i:nice:ii}  $k$ is a complete Pick kernel;
\item \label{i:nice:iii} $\CM(k)$ has the Douglas property.
\end{enumerate}
The row $q$ in the niceness representation may have infinitely
many components.
\end{theorem}

\begin{proof}
That item~\ref{i:nice:ii} implies item~\ref{i:nice:iii}  follows from
Theorem~\ref{thm:cp-douglas-countable}; and  
that  item~\ref{i:nice:iii} implies item~\ref{i:nice:i} follows by the corner-compression argument
after Definition~\ref{def:douglas}.

To complete the proof, assume item~\ref{i:nice:i} holds.   Choose a niceness representation in which the row
$p$ has $r$ components and the row $q$ is countable.
As in equation~\eqref{eq:nice}, the identity
\[
 1=k(z,w)\bigl[p(z)p(w)^*-q(z)q(w)^*\bigr]
\]
implies $M_pM_p^*\succeq M_qM_q^*$.
Lemma~\ref{lem:finite-left-row}, applied to
$\mathcal A=\CM(k)$, yields a contractive operator with
entries in $\CM(k)$ satisfying $M_pC=M_q$.
Its finite column truncations are rectangular matrix multipliers,
and converge strongly to $C$.  By
Lemma~\ref{lem:op-mults-plus}, $C=M_\Phi$ for some
$\Phi\in\cM{\ell^2}{\mathbb C^r}$, with
\[
 p\Phi=q,\qquad \|M_\Phi\|\leq1.
\]
Pad $p$ and $\Phi$ with zero components and zero rows,
respectively, to give the coefficient spaces used in
Section~\ref{sec:main}.  The sole use of the Douglas property
in that section was to obtain this contractive factorization.
The projection calculation there therefore applies and produces,
for every $\omega\in\Omega$, a positive semi-definite kernel
$\widetilde L_\omega$ satisfying~\eqref{eq:np}.
Consequently $k$ is complete Pick, proving that item~\ref{i:nice:i} implies item~\ref{i:nice:ii}
and completing the proof of the Theorem.
\end{proof}

\begin{remark}\rm
\label{rem:finite-p-scope}
The hypothesis on $p$ is satisfied whenever $1/k$ is a finite
difference of sums of squares of multipliers.  In particular,
it applies to the standard ball and polydisk kernels considered
above, whose reciprocals are polynomials in the coordinate
products.  Expanding these polynomials separates their positive
and negative coefficients, and every resulting monomial is a
bounded multiplier.  The conclusions of this section do not
resolve the finite-to-countable implication for arbitrary scalar
multiplier algebras; see Remark~\ref{rem:finite-commutative}.
\end{remark}

\appendix

\section{Summary of substantive revisions}
\label{app:changes}

This appendix records the material differences between the
present text and the supplied version,
\texttt{smtt01nov11.tex}.  The numbering of results in
Sections~1--5 has been retained.
Sections~\ref{app:finite-matrices}
and~\ref{sec:finite-countable} contain the new counterexample
and the results on finite and countable factorizations,
respectively.

\begin{enumerate}
\item \label{i:no-finite-to-infinite} 
The claim following Definition~\ref{def:douglas} in the
November 2011 arXiv version of the paper, that finite-matrix
factorization implies the Douglas property, has been removed.
Section~\ref{app:finite-matrices} supplies a wot-closed unital
noncommutative algebra for which every finite problem has a
contractive solution, while a countably amplified problem has no
solution at all.  Remark~\ref{rem:app-cstar-topologies} identifies the generated
$C^*$-algebras and their Calkin images, and proves that the
canonical $*$-isomorphism and its inverse are neither sot
nor wot continuous. Remark~\ref{rem:finite-commutative} explains
that this construction does not settle the corresponding
question for arbitrary commutative wot-closed unital algebras.
For nice kernels, Section~\ref{sec:finite-countable} gives an
affirmative answer when the row $p$ can be chosen with
finitely many components.

\item 
The argument proving Theorem~\ref{thm:main} is unchanged as it did not use
the erroneous finite-to-infinite implication of item~\eqref{i:no-finite-to-infinite}.
The final positive kernel is explicitly
rescaled by $k(\omega,\omega)$ to obtain the normalization in
Definition~\ref{def:NPk}.  

\item 
Theorem~\ref{thm:main-aux} now places both its data and its putative
solution in $\mathcal A\otw B(\ell^2)$.  This is the space containing
the infinite block operators constructed in the proof; the
algebraic tensor product in the previous statement did not express
that construction.  The revised theorem asserts the existence of
data for which there is no bounded multiplier solution in the wot
amplification, even without a contractive norm requirement.

\item 
The dimensions and indices of the finite blocks in
Section~\ref{sec:main-aux} have been corrected.  The norm estimate
is now $\|C_N\|\geq\sqrt{N+1},$ corrected from $N+1.$ This correction does
not materially affect the result. 

\item 
The proofs for higher dimensional polydisks and balls are now
included.  The polydisk argument uses tensor products and
restriction of multiplier symbols to the first two coordinates.
For the ball, a positive-kernel expansion for
$(1-\langle z,w\rangle)^{-\alpha}$ establishes
equation~\eqref{eq:ball-coordinate-inequality}, which supplies
$2\times2$ blocks with the required growing solution norms.
The choices $\alpha=d$ and $\alpha=d+1$ cover the stated Hardy
and Bergman ball spaces, respectively.

\item 
Example~\ref{ex1} has been expanded. It retains its kernel but replaces the inference
from failure of the complete Pick property with an explicit
$2\times2$ multiplier factorization problem.  That problem has
no solution of norm less than $\sqrt{2}$.  The kernel inequalities
$s\preceq k\preceq2s$, equality of multiplier norms with the
Hardy space at every coefficient level, and Leech's theorem show
that every admissible problem has a solution of norm at most
$\sqrt{2}$.  This gives the exact uniform constant, including
for countable amplifications.

\item 
Section~\ref{sec:finite-countable} has been added.
Theorem~\ref{thm:cp-douglas-countable} obtains the full Douglas
property for separable complete Pick spaces from the operator-valued
factorization theorem of Ball, Trent, and Vinnikov, with orthogonal
components treated separately when the kernel has zeros.
Lemma~\ref{lem:finite-left-row} supplies the compactness argument
for a fixed finite left row and a countable right row.
Theorem~\ref{thm:nice-finite-p} then shows that, for nice kernels
admitting a representation with finitely many components in $p$,
the finite-matrix factorization property, the complete Pick property,
and the full Douglas property are equivalent.  The introduction
now describes these conclusions and points to the new section.

\item 
Problem~\ref{prob:bounded-douglas} has been corrected. It now states the uniform bound
with the intended existential quantifiers: a constant depending
only on the algebra bounds the norm of a suitable solution for
each pair.  A direct-sum argument proves this assertion from the
bounded Douglas property.  Example~\ref{ex:not-nice} now verifies
explicitly that the Fock kernel is not complete Pick and is not
nice, while its multiplier algebra consists of constants and has
the Douglas property.

\item 
There are also updates and minor fixes  to the preliminaries, notation, and references.
The domain and codomain of the kernel-vector formula in
Lemma~\ref{lem:eigs-adjoints} have been corrected.
Lemma~\ref{lem:op-mults-plus} states its nonvanishing-diagonal
and separability hypotheses and gives a direct wot-closure proof.
Rank-one operator notation in Section~\ref{sec:main-aux} is
distinguished from Hilbert-space tensor products, and the
projection identities, adjoint powers, and summation indices are
corrected.  Typographical and bibliographic corrections include
the published reference for Leech's theorem and the addition of
the commutant-lifting reference used in
Section~\ref{app:finite-matrices}.
\end{enumerate}

\bigskip

\linespread{1.1}

\printindex

\end{document}